\documentclass[final,11pt]{siamltex}

\usepackage{graphicx}
\usepackage{color}
\usepackage{bm}
\usepackage{amsfonts}

\newcommand{\tr}{^{\sf T}}

\newcommand{\C}[1]{{\cal {#1}}}
\newcommand{\Supp}{\mbox{supp}}

\input epsf

\newtheorem{remark}{Remark}

\title{
Projection Algorithms for Nonconvex Minimization
with Application to Sparse Principal Component Analysis
\thanks{Author names are ordered alphabetically.
March 28, 2015, revised December 25, 2015.
The authors gratefully acknowledge support by
National Science Foundation under grants
1115568 and 1522629 and by Office of Naval Research under
grants N00014-11-1-0068 and N00014-15-1-2048.} }

\author{
        William W. Hager$^*$\thanks{{\tt hager@ufl.edu},
        http://people.clas.ufl.edu/hager/,
        PO Box 118105,
        Department of Mathematics,
        University of Florida, Gainesville, FL 32611-8105.
        Phone (352) 294-2308. Fax (352) 392-8357.}
\and
        Dzung T. Phan$^*$\thanks{{\tt phandu@us.ibm.com},
        IBM T.~J.~Watson Research Center,
        Yorktown Heights, NY 20598.
        Phone (914) 945-1883.}
\and
        Jiajie Zhu$^*$\thanks{{\tt zplusj@ufl.edu},
        http://people.clas.ufl.edu/zplusj/
        PO Box 118105,
        Department of Mathematics,
        University of Florida, Gainesville, FL 32611-8105.
        Currently in the Max Planck Institute for Intelligent Systems, Germany.}
}

\begin{document}

\maketitle

\begin{abstract}
We consider concave minimization problems over nonconvex sets.
Optimization problems with this structure arise in sparse
principal component analysis.
We analyze both a gradient projection algorithm and an approximate
Newton algorithm where the Hessian approximation is a multiple of the identity.
Convergence results are established.
In numerical experiments arising in sparse principal component analysis,
it is seen that the performance of the
gradient projection algorithm is very similar to that of the truncated
power method and the generalized power method.
In some cases, the approximate Newton algorithm with
a Barzilai-Borwein (BB) Hessian approximation and a nonmonotone line search
can be substantially faster than the other algorithms, and can converge
to a better solution.
\end{abstract}

\begin{keywords}
sparse principal component analysis, gradient projection,
nonconvex minimization, approximate Newton, Barzilai-Borwein method
\end{keywords}


\pagestyle{myheadings} \thispagestyle{plain}
\markboth{W. W. HAGER, D. T. PHAN,  AND J. J. ZHU}
{PROJECTION ALGORITHMS FOR NONCONVEX MINIMIZATION}

\section{Introduction}
Principal component analysis (PCA) is an extremely popular
tool in engineering and statistical analysis.
It amounts to computing the singular vectors associated with the
largest singular values.
In its simplest setting, the rank-one approximation,
amounts to solving an optimization problem of the form
\begin{equation} \label{PCA}
\max \{ x\tr \Sigma x : x \in \mathbb{R}^n, \quad \|x\| = 1 \},
\end{equation}
where $\Sigma = A\tr A$ is the covariance matrix associated with the
data matrix $A \in \mathbb{R}^{m \times n}$ and $\| \cdot \|$ is the
Euclidean norm.
As pointed out in \cite{luss2013},
there is no loss of generality in assuming that $\Sigma$ is
positive definite since
\[
x\tr \Sigma x + \sigma = x\tr (\Sigma + \sigma I) x
\]
whenever $x$ is feasible in (\ref{PCA}).

The lack of interpretability has been a major concern in PCA.
Sparse PCA partly addresses this problem by constraining
the number of nonzero components of the maximizing $x$ in (\ref{PCA}).
Given a positive integer $\kappa$,
the sparse PCA problem associated with (\ref{PCA}) is
\begin{equation} \label{sparsePCA}
\max \{ x\tr\Sigma x : x \in \mathbb{R}^n, \quad
\|x\| =1, \quad \|x\|_0 \leq \kappa \},
\end{equation}
where $\| x \|_0$ denotes the number of nonzero components of $x$.
Due to the sparsity constraint in (\ref{sparsePCA}), the
feasible set is no longer convex, which makes the
optimization problem more difficult.

A large quantity of literatures focus on using the sparsity norms, such as the $l_1$-norm and $l_0$-norm, in the formulation. Such methods are popular in signal processing and statistics community. Notably, the lasso-related methods appeared in studies such as \cite{chen1998atomic,efron2004least,donoho2006compressed}.
In \cite{cadima1995loading} PCA loadings smaller than a certain tolerance are
simply set to zero to produce sparse principal components.
More recently, optimization-based approaches have been used to introduce sparsity.
For example, in \cite{lt11convex} sparsity is achieved using an
$l_1$ relaxation.
That is, the problem (\ref{sparsePCA}) is replaced by
\begin{equation} \label{L1PCA}
\max \{ x\tr\Sigma x : x \in \mathbb{R}^n, \quad
\|x\| =1, \quad \|x\|_1^2 \leq \kappa \},
\end{equation}
where $\| x \|_1 = |x_1| + |x_2| + \ldots + |x_n|$.
The solution of the relaxed problem (\ref{L1PCA}) yields an upper bound
for the solution of (\ref{sparsePCA}).
In \cite{jolliffe03} the Rayleigh quotient problem subject to an
$l_1$-constraint is successively maximized using the authors'
SCoTLASS algorithm.
In \cite{zou06spca}, 
the authors formulate a regression problem and propose numerical algorithms
to solve it.
Their approach can be applied to large-scale data, but it is computationally
expensive.
In \cite{d08optimal} a new semi-definite relaxation is formulated
and a greedy algorithm is developed that computes a full set of good
solutions for the target number of non-zero coefficients.
With total complexity $O(n^3)$, the algorithm is computationally expensive.
Other references related to sparse optimization include
\cite{Clarkson10, Hazan12, Jaggi13, jenatton2010,Lacoste12, Sriperumbudur11}.

Our work is largely motivated by \cite{journee2010},
\cite{luss2013}, and \cite{yuan2013}.
In \cite{journee2010} both $l_1$-penalized and $l_0$-penalized
sparse PCA problems are considered and
a generalized power method is developed.
The numerical experiments show that their approach outperforms
earlier algorithms both in solution quality and in computational speed.
Recently, \cite{luss2013} and \cite{yuan2013} both consider the
$l_0$-constrained sparse PCA problem and propose
an efficient truncated power method.
Their algorithms are equivalent and originated
from the classic Frank-Wolfe \cite{fra-wol:56} conditional gradient algorithm.

In this paper, we study both the gradient projection algorithm and
an approximate Newton algorithm.
Convergence results are established and numerical experiments are
given for sparse PCA problems of the form (\ref{sparsePCA}).
The algorithms have the same iteration complexity
as the fastest current algorithms, such as the algorithms proposed in \cite{journee2010}, \cite{luss2013} and \cite{yuan2013}, which is $O(\kappa n)$ in solving problem (\ref{sparsePCA}).
The gradient projection algorithm with unit step size has nearly
identical performance as that of the conditionl gradient algorithm with unit step-size (ConGradU) \cite{luss2013} and the truncated power method (Tpower) \cite{yuan2013}.
On the other hand, the approximate Newton algorithm
can often converge faster to a better objective value than the other
algorithms.

The paper is organized as follows.
In Section~\ref{GP} we analyze the gradient projection algorithm when
the constraint set is nonconvex.
Section~\ref{Scale} introduces and analyzes the approximate Newton scheme.
Section~\ref{numerical}
examines the performance of the algorithms
in some numerical experiments based on
classic examples found in the sparse PCA literature.

{\bf Notation.} If $f:R^n\to R$ is differentiable,
then $\nabla f (x)$ denotes the gradient of $f$, a row vector,
while $g (x)$ denotes the gradient of $f$ arranged as a column vector.
The subscript $k$ denotes the iteration number.
In particular, $x_k$ is the $k$-th $x$-iterate and $g_k = g(x_k)$.
The $i$-th element of the $k$-th iterate is denoted $x_{ki}$.
The Euclidean norm is denoted $\|\cdot \|$, while
$\|\cdot \|_0$ denotes cardinality (number of non-zero elements).
If $x \in \mathbb{R}^n$, then the support of $x$ is the set of
indices of nonzeros components:
\[
\Supp (x) = \{ i : x_i \ne 0 \} .
\]
If $\Omega \subset \mathbb{R}^n$, then $\mbox{conv}(\Omega)$ is
the convex hull of $\Omega$.
If $\C{S} \subset \{1, 2, \ldots, n \}$, then $x_\C{S}$ is
the vector obtained by replacing $x_i$ for $i \in \C{S}^c$ by 0.
If $\C{A}$ is a set, then $\C{A}^c$ is its complement.
%
\section{Gradient projection algorithm}
\label{GP}
Let us consider an optimization problem of the form
\begin{equation}\label{pb}
\min \{ f(x) : x \in \Omega \},
\end{equation}
where $\Omega \subset \mathbb{R}^n$ is a nonempty, closed set and
$f : \Omega \to \mathbb{R}$ is differentiable on $\Omega$.
Often, the gradient projection algorithm is presented in the
context of an optimization problem where the feasible set $\Omega$ is convex
\cite{beck2003mirror,be82,hz05a}.
Since the feasible set for the sparse PCA problem (\ref{sparsePCA})
is nonconvex, we will study the gradient projection algorithm for
a potentially nonconvex feasible set $\Omega$.

The projection of $x$ onto $\Omega$ is defined by
\[
P_\Omega (x) = \arg\min_{y\in \Omega} \| x-y\|.
\]
For the constraint set $\Omega$ that arises in sparse PCA,
the projection can be expressed as follows:
\smallskip

\begin{proposition}
\label{projOmega}
For the set
\begin{equation}\label{omega}
\Omega = \{ x \in \mathbb{R}^n :
\|x\| =1, \quad \|x\|_0 \leq \kappa \},
\end{equation}
where $\kappa$ is a positive integer, we have
\[
T_\kappa (x)/\|T_\kappa (x)\| \in P_\Omega (x),
\]
where $T_\kappa (x)$ is the vector obtained from $x$ by replacing
$n - \kappa$ elements of $x$ with smallest magnitude by $0$.
\end{proposition}
\begin{proof}
If $y \in \Omega$, then $\|x - y\|^2 = \|x\|^2 + 1 - 2\langle x, y \rangle$.
Hence, we have
\begin{equation}\label{s0}
P_\Omega (x) = \arg \; \min \{ -\langle x , y \rangle : \|y\| = 1,
\quad \|y\|_0 \le \kappa \}.
\end{equation}
In \cite[Prop. 4.3]{luss2013}, it is shown that the minimum is attained
at $y = T_\kappa (x)/\|T_\kappa (x)\|$.
We include the proof since it is short and we need to refer to it later.
Given any set $\C{S} \subset \{1, 2, \ldots , n\}$, the solution of the
problem
\[
\min \{ - \langle x , y \rangle : \|y\| = 1, \quad \Supp (y) = \C{S} \}
\]
is $y = x_\C{S}/\|x_\C{S}\|$ by the Schwarz inequality, and the
corresponding objective value is $-\|x_\C{S}\|$, where $x_\C{S}$ is
the vector obtained by replacing $x_i$ for $i \in \C{S}^c$ by 0.
Clearly, the minimum is attained when $\C{S}$ is the set of
indices of $x$ associated with the $\kappa$ absolute largest components.
\end{proof}

In general, when $\Omega$ is closed, the projection exists,
although it may not be unique when $\Omega$ is nonconvex.
If $x_k \in \Omega$ is the current iterate,
then in one of the standard implementations
of gradient projection algorithm, $x_{k+1}$ is obtained by a line search
along the line segment connecting $x_k$ and
$P_\Omega (x_{k}-s_k g_k)$, where $g_k$ is the gradient at $x_k$ and
$s_k > 0$ is the step size.
When $\Omega$ is nonconvex, this line segment is not always contained
in $\Omega$.
Hence, we will focus on gradient projection algorithms of the form
\begin{equation}\label{gradproj}
x_{k+1} \in P_\Omega (x_{k}-s_k g_k) .
\end{equation}
Since $\Omega$ is closed, $x_{k+1}$ exists for each $k$.
We first observe that $x_{k+1} - x_k$ always forms an
obtuse angle with the gradient, which guarantees descent when $f$ is concave.
\smallskip

\begin{lemma}
\label{descent}
If $x_k \in \Omega$, then
\begin{equation}\label{p1}
\nabla f(x_k) (y - x_k) \le 0 \quad \mbox{for all }
y \in P_\Omega (x_{k}-s_k g_k) .
\end{equation}
In particular, for $y = x_{k+1}$ this gives
\begin{equation}\label{<=0}
\nabla f(x_k) (x_{k+1} - x_k) \le 0 .
\end{equation}
If $f$ is concave over {\rm conv}$(\Omega)$, then
$f(x_{k+1}) \le f(x_k)$.
\end{lemma}
\begin{proof}
If $y \in P_\Omega (x_k - s_k g_k)$,
then since $P_\Omega (x_k - s_k g_k)$ is set of
elements in $\Omega$ closest to $x_k - s_k g_k$,
we have
\begin{equation} \label{xk+1}
\|y - (x_k - s_k g_k)\| \le \|x_k -(x_k - s_k g_k)\| = s_k\|g_k\| .
\end{equation}
By the Schwarz inequality and (\ref{xk+1}), it follows that
\[
g_k\tr (y - (x_k - s_k g_k)) \le
\|g_k\| \|y - (x_k - s_k g_k) \| \le s_k \|g_k\|^2 .
\]
We rearrange this inequality to obtain (\ref{p1}).
If $f$ is concave over $\mbox{conv}(\Omega)$, then
\begin{equation}\label{concave}
f(x_{k+1}) \le f(x_k) + \nabla f(x_k) (x_{k+1} - x_k) .
\end{equation}
By (\ref{<=0}), $f(x_{k+1}) \le f(x_k)$.
\end{proof}

The following result is well known.
\smallskip

\begin{proposition}
\label{vi}
If $f:R^n\to R$ is concave and $\Omega \subset \mathbb{R}^n$, then
\begin{equation}\label{conv}
\inf\{f(x) | x\in \Omega \}=\inf\{f(x) | x\in \mbox{\rm conv}(\Omega) \},
\end{equation}
where the first infimum is attained only when the second infimum is attained.
Moreover, if $f$ is differentiable at
$x^* \in \arg \; \min \{f(x) : x \in \Omega\}$, then
\begin{equation} \label{stnry}
\nabla f(x^*) (y-x^*) \ge 0 \quad
\mbox{for all } y\in \mbox{\rm conv} (\Omega).
\end{equation}
\end{proposition}
\begin{proof}
The first result (\ref{conv}) is proved in \cite[Thm.~32.2]{Rockafellar70}.
If $x^*$ minimizes $f(x)$ over $\Omega$, then by (\ref{conv}),
\[
x^* \in \arg \; \min \{f(x) : x \in \mbox{conv}(\Omega)\}.
\]
Since $\mbox{conv}(\Omega)$ is a convex set,
the first-order optimality condition for $x^*$ is (\ref{stnry}).
\end{proof}

\begin{remark}
Note that at a local minimizer $x^*$ of $f$ over a nonconvex set
$\Omega$, the inequality $\nabla f (x^*) (y-x^*) \ge 0$
may not hold for all $y \in \Omega$.
For example, suppose that $f(x) = a\tr x$ where $\nabla f = a$ has
the direction shown in Figure~$\ref{ctex}$.
The point A is a local minimizer of $f$ over $\Omega$, but
($\ref{stnry}$) does not hold.
Hence, Proposition~$\ref{vi}$ is only valid for a global minimizer, as stated.
\end{remark}
\begin{figure}
\begin{center}
\includegraphics[scale=.3]{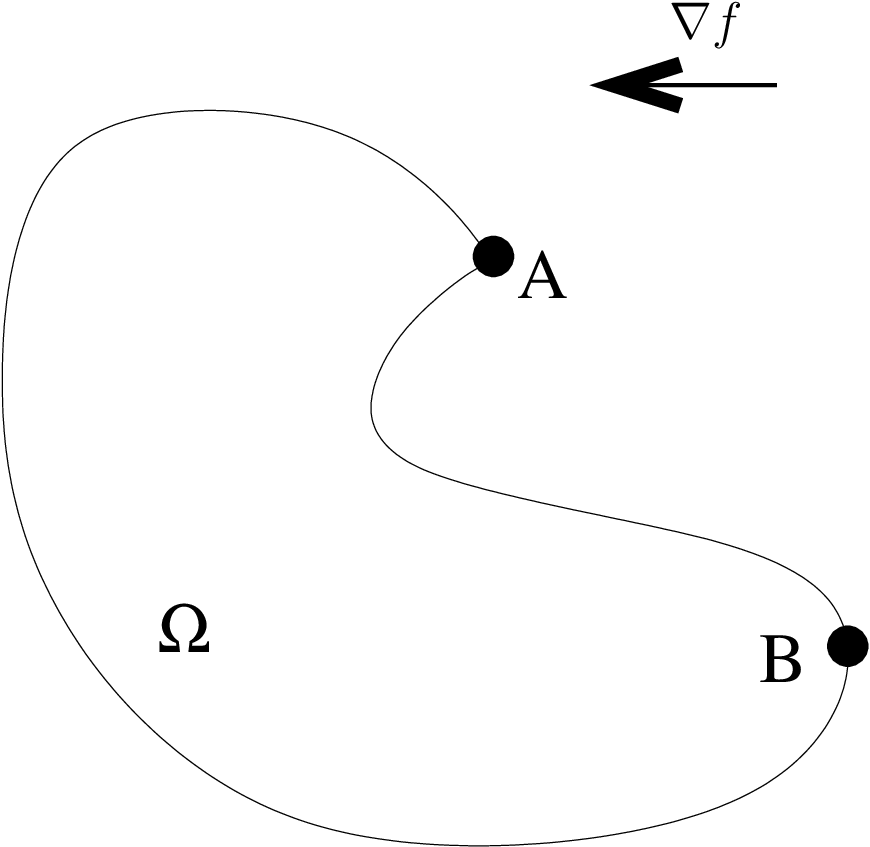}
\caption{Example that shows Proposition~\ref{vi} may not hold for local
minimizers.}
\label{ctex}
\end{center}
\end{figure}

Next, we consider the special choice
$y \in P_\Omega (x^* - s g(x^*))$ in Proposition~\ref{vi}.
\smallskip

\begin{corollary}\label{=0}
If $f:R^n\to R$ is concave and
\[
x^* \in \arg \; \min \{f (x): x \in \Omega\},
\]
then
\begin{equation}\label{=0eq}
\nabla f(x^*) (y - x^*) = 0
\end{equation}
whenever $y \in P_\Omega (x^* - s g(x^*))$ for some $s \ge 0$.
\end{corollary}
\begin{proof}
By Proposition~\ref{vi}, we have
\[
\nabla f(x^*) (y -x^*)  \ge 0
\]
for all $y \in P_\Omega (x^* -s g(x^* ))$.
On the other hand, by Lemma~\ref{descent} with $x_k = x^*$, we have
\[
\nabla f(x^*) (y -x^*)  \le 0
\]
for all $y \in P_\Omega (x^* -s g(x^* ))$.
Therefore, (\ref{=0eq}) holds.
\end{proof}

The following property for the projection is needed in the main theorem:
\smallskip

\begin{lemma} \label{continuity}
If $\Omega$ is a nonempty closed set, $x_k \in \mathbb{R}^n$ is a
sequence converging to $x^*$ and $y_k \in P_\Omega (x_k)$ is a
sequence converging to $y^*$, then $y^* \in P_\Omega (x^*)$.
\end{lemma}
\begin{proof}
Since $y_k \in \Omega$ for each $k$ and $\Omega$ is closed,
$y^* \in \Omega$.
Hence, we have
\[
\|y^* - x^*\| \ge \min_{y \in \Omega} \|y - x^*\|.
\]
If this inequality is an equality, then we are done;
consequently, let us suppose that
\[
\|y^* - x^*\| > \min_{y \in \Omega} \|y - x^*\| \ge
\min_{y \in \Omega} \{ \|y - x_k \| - \|x_k - x^*\| \} =
\|y_k - x_k \| - \| x_k - x^*\| .
\]
As $k$ tends to $\infty$, the right side approaches
$\|y^* - x^*\|$, which yields a contradiction.
\end{proof}

We now give further justification for the convergence
of the gradient projection algorithm in the nonconvex setting.
\smallskip

\begin{theorem}
\label{converge}
If $f:R^n\to R$ is concave,
$\Omega$ is a compact nonempty set,
and $x_k$ is generated by the gradient projection algorithm
$(\ref{gradproj})$, then we have $f(x_{k+1}) \le f(x_k)$ for each $k$ and
\begin{equation}\label{liminf}
\lim\limits_{k \to\infty} \nabla f(x_k) (x_{k+1} - x_k) = 0.
\end{equation}
If $x^*$ is the limit of any convergent subsequence of the $x_k$
and the step size $s_k$ approaches a limit $s^*$, then
\begin{equation}\label{s1}
\nabla f(x^*) (y - x^*) \le 0 \quad \mbox{for all }
y \in P_\Omega (x^* - s^* g(x^*)).
\end{equation}
If $f$ is continuously differentiable around $x^*$, then
\begin{equation}\label{s2}
\nabla f(x^*) (y^* - x^*) = 0
\end{equation}
for some $y^* \in P_\Omega (x^* - s^* g(x^*))$.
\end{theorem}
\begin{proof}
Sum the concavity inequality (\ref{concave}) for $k = 0, 1, \ldots,$ $K-1$
to obtain
\begin{equation}\label{e1}
{f(x_{K})}-{f(x_{0})}\leq \sum_{k=0}^{K-1}
{\nabla f(x_k)} ({x_{k+1}} - {x_k}).
\end{equation}
Since $f$ is continuous and $\Omega$ is compact,
$f^* = \min \{f(x): x \in \Omega\}$ is finite and
\begin{equation}\label{e2}
f^* - f(x_0) \le f(x_K) - f(x_0) .
\end{equation}
Together, (\ref{e1}) and (\ref{e2}) yield (\ref{liminf}) since
$\nabla f(x_k) (x_{k+1} - x_k) \le 0$ for each $k$ by
Lemma~\ref{descent}.

The relation (\ref{s1}) is (\ref{p1}) with $x_k$ replaced by $x^*$.
For convenience, let $x_k$ also denote the subsequence of the iterates
that converges to $x^*$, and let $y_k \in P_\Omega (x_k - s_k g_k)$
denote the iterate produced by $x_k$.
Since $y_k$ lies in a compact set,
there exists a subsequence converging to a limit $y^*$.
Again, for convenience, let $x_k$ and $y_k$ denote this convergent
subsequence.
By (\ref{liminf}) and the fact that $f$ is continuously differentiable
around $x^*$, we have
\[
\lim\limits_{k \to\infty} \nabla f(x_k) (y_{k} - x_k) =
\nabla f(x^*) (y^* - x^*) = 0 .
\]
By Lemma~\ref{continuity}, $y^* \in P_\Omega (x^* - s^* g(x^*))$.
\end{proof}.

\begin{remark}
The inequalities ($\ref{<=0}$), ($\ref{e1}$), and ($\ref{e2}$) imply that
\[
\min_{0 \le k \le K} \nabla f(x_k) (x_k - x_{k+1}) \le
\frac{f(x_0) - f^*}{K+1} .
\]
\end{remark}

When $\Omega$ is convex, much stronger convergence results can be
established for the gradient projection algorithm.
In this case, the projection onto $\Omega$ is unique.
%
By \cite[Prop. 2.1]{hz05a}, for any $x \in \Omega$ and $s > 0$,
$x = P_\Omega (x - s g(x))$ if and only if $x$ is a stationary point
for (\ref{pb}).
That is,
\[
\nabla f (x) (y - x) \ge 0 \quad \mbox{for all } y \in \Omega .
\]
Moreover, when $\Omega$ is convex,
\begin{equation}\label{convex}
\nabla f (x) (P_\Omega (x - s g(x)) - x) \le
-\|P_\Omega (x - \alpha g(x)) - x \|^2/s
\end{equation}
for any $x \in \Omega$ and $s > 0$.
Hence, (\ref{s2}) implies that the left side of (\ref{convex}) vanishes
at $x = x^*$, which means that $x^* = P_\Omega (x^* - s g(x^*))$.
And conversely, if $x^* = P_\Omega (x^* - s g(x^*))$, then (\ref{s2}) holds.

\section{Approximate Newton algorithm}
\label{Scale}
To account for second-order information,
Bertsekas \cite{be82} analyzes the following version of the
gradient projection method:
\[
x_{k+1} \in P_\Omega (x_{k}-s_k \nabla^2 f(x_k)^{-1} g_k) .
\]
Strong convergence results can be established when $\Omega$ is convex and
$f$ is strongly convex.
On the other hand, if $f$ is concave, local minimizers are extreme points
of the feasible set, so the analysis is quite different.
Suppose that $\nabla^2 f(x_k)$ is approximated by a multiple
$\alpha_k$ of the identity matrix as is done in the BB method \cite{bb88}.
This leads to the approximation
\begin{equation}\label{Qmodel}
f(x) \approx f(x_k) + \nabla f(x_k) (x - x_k) + \frac{\alpha_k}{2} \|x-x_k\|^2 .
\end{equation}
Let us consider the algorithm in which the new iterate $x_{k+1}$ is
obtained by optimizing the quadratic model:
\begin{equation}\label{QP}
x_{k+1} \in \arg \; \min \{ \nabla f(x_k) (x - x_k)
+ \frac{\alpha_k}{2} \|x-x_k\|^2 : x \in \Omega\}
\end{equation}
After completing the square, the iteration is equivalent to
\[
x_{k+1} \in \arg \; \min \left\{ \alpha_k
\|x - (x_k - g_k/\alpha_k) \|^2 : x \in \Omega \right\} .
\]
If $\alpha_k > 0$, then this reduces to
$x_{k+1} \in P_\Omega (x_k - g_k/\alpha_k)$;
in other words, perform the gradient projection algorithm with
step size $1/\alpha_k$.
If $\alpha_k < 0$, then the iteration reduces to
\begin{equation}\label{max}
x_{k+1} \in Q_\Omega (x_k - g_k /\alpha_k) ,
\end{equation}
where
\begin{equation}\label{Q}
Q_\Omega (x) = \arg \; \max \{ \|x - y\| : y \in \Omega \}.
\end{equation}
If $\Omega$ is unbounded, then this iteration does not make sense
since the maximum occurs at infinity.
But if $\Omega$ is compact, then the iteration is justified in the sense
that the projection (\ref{Q}) exists and the iteration
is based on a quadratic model of the function, which could be
better than a linear model.

Suppose that $x^* \in Q_\Omega (x^* - g(x^*))/\alpha)$
for some $\alpha < 0$.
Due to the equivalence between (\ref{QP}) and (\ref{max}), it follows that
\begin{equation}\label{x^*}
x^* \in \arg \; \min \{ \nabla f(x^*) (x - x^*)
+ \frac{\alpha}{2} \|x-x^*\|^2 : x \in \Omega\} .
\end{equation}
That is, $x^*$ is a global optimizer of the quadratic objective in (\ref{x^*})
over $\Omega$.
Since the objective in the optimization problem (\ref{x^*}) is concave,
Proposition~\ref{vi} yields
\begin{equation}\label{nc}
\nabla f(x^*)(y - x^*) \ge 0 \quad
\mbox{for all } y\in \mbox{\rm conv} (\Omega).
\end{equation}
Hence, fixed points of the iteration (\ref{max}) satisfy the necessary
condition (\ref{nc}) associated with a global optimum of (\ref{pb}).

In the special case where $\Omega$ is the constraint set (\ref{omega})
appearing in sparse PCA and $\alpha_k < 0$, the maximization in
(\ref{Q}) can be evaluated as follows:
\smallskip

\begin{proposition}
\label{antiprojOmega}
For the set $\Omega$ in $(\ref{omega})$ associated with sparse PCA, we have
\[
-T_\kappa (x)/\|T_\kappa (x)\| \in Q_\Omega (x).
\]
\end{proposition}
\begin{proof}
As in the proof of Proposition~\ref{projOmega},
$\|x - y\|^2 = \|x\|^2 + 1 - 2\langle x, y \rangle$ when $y$ lies in the
set $\Omega$ of (\ref{omega}).
Hence, we have
\[
Q_\Omega (x) = \arg \; \max \{ -\langle x , y \rangle : \|y\| = 1,
\quad \|y\|_0 \le \kappa \}.
\]
Given any set $\C{S} \subset \{1, 2, \ldots , n\}$, the solution of the
problem
\[
\max \{ - \langle x , y \rangle : \|y\| = 1, \quad \Supp (y) = \C{S} \}
\]
is $y = -x_\C{S}/\|x_\C{S}\|$ by the Schwarz inequality, and the
corresponding objective value is $\|x_\C{S}\|$.
Clearly, the maximum is attained when $\C{S}$ corresponds to a set of
indices of $x$ associated with the $\kappa$ absolute largest components.
\end{proof}

Let us now study the convergence of the
iterates generated by the quadratic model (\ref{QP}).
The case where $\alpha_k > 0$ corresponds to the gradient projection
algorithm which was studied in Section~\ref{GP}.
In this section, we focus on $\alpha_k < 0$ and
the iteration $x_{k+1} \in Q_\Omega (x_k - g_k/\alpha_k)$.
For the numerical experiments, we employ a
BB-approximation \cite{bb88} to the Hessian given by
\begin{equation}\label{BB}
\alpha_k^{BB} =
\frac{(\nabla f(x_k) - \nabla f (x_{k-1}))(x_k -x_{k-1})}
{\|x_k - x_{k-1}\|^2} .
\end{equation}
%
It is well known that the BB-approximation performs much better
when it is embedded in a nonmonotone line search.
This leads us to study a scheme based on the GLL stepsize rule of \cite{gll86}.
Let $f^{\max}_k$ denote the largest of the $M$ most recent function values:
\[
f^{\max}_k = \max \{ f(x_{k-j}) : 0 \le j < \min(k,M) \}.
\]
Our convention is that $f_k^{\max} = \infty$ when $M = 0$.
The nonmonotone approximate Newton algorithm that we analyze is as follows:
\bigskip

{\tt
\begin{tabular}{l}
{\sc Nonmonotone Approximate Newton (for strongly concave $f$)}\\
\hline
Given $\sigma \in (0,1)$,
$[\alpha_{\min}, \alpha_{\max}] \subset (-\infty, 0)$,
and starting guess $x_0$. \\
Set $k = 0$.\\
\begin{tabular}{rl}
Step 1. & Choose $\beta_k \in [\alpha_{\min} , \alpha_{\max} ]$\\
Step 2. & Set $\alpha_k = \sigma^j \beta_k$
where $j \ge 0$ is the smallest integer \\
& such that \\[.1in]
& $f (x_{k+1}) \le f^{\max}_k + (\alpha_k/2)
\|{x}_{k+1} - {x}_k\|^2 $ where\\[.1in]
& ${x}_{k+1} \in Q_\Omega (x_k - g_k/\alpha_k)$\\[.1in]
Step 3. & If a stopping criterion is satisfied, terminate. \\
Step 4. & Set $k=k+1$ and go to step 1. \\
\end{tabular}\\
\hline \\
\end{tabular}}
\bigskip

Note that the approximate Newton algorithm is monotone when the memory $M = 1$.
If $M = 0$, then the stepsize acceptance condition in Step~2 is satisfied
for $j = 0$ and $\alpha_k = \beta_k$.
Hence, when $M = 0$ there is no line search.
To ensure that $\beta_k$ lies in the
safe-guard interval $[\alpha_{\min}, \alpha_{\max}]$ when using
the BB-approximation to the Hessian, it is common to take
\[
\beta_k = \mbox{mid }\{ \alpha_{\min} , \alpha_k^{BB}, \alpha_{\max} \},
\]
where mid denotes median or middle.
In the numerical experiments, $\alpha_{\min}$ is large in magnitude
and $\alpha_{\max}$ is close to 0, in which case
the safe guards have no practical significance; nonetheless, they enter into
the convergence theory.

In the approximate Newton algorithm, we make a starting
guess $\beta_k$ for the initial $\alpha_k$
and then increase $\alpha_k$ until the line search
acceptance criterion of Step~2 is fulfilled.
We first show that for a strongly concave function, the line
search criterion is satisfied for $j$ sufficiently large,
and the stepsize is uniformly bounded away from zero.
\smallskip

\begin{proposition} \label{stepsize}
If $f$ is differentiable on a compact set $\Omega\subset \mathbb{R}^n$
and for some $\mu < 0$, we have
\begin{equation}\label{concave3}
f(y) \le f(x) + \nabla f(x) (y-x) + \frac{\mu}{2}\|y-x\|^2 \quad
\mbox{for all } x \mbox{ and } y \in \Omega,
\end{equation}
then Step~$2$ in the approximate Newton algorithm terminates
with a finite $j$, and with $\alpha_k$ bounded away from $0$, uniformly in $k$.
In particular, we have $\alpha_k \le$ $\bar{\alpha} :=$
$\max (\sigma \mu/2,\alpha_{\max}) < 0$.
\end{proposition}
\smallskip

\begin{proof}
Since $\Omega$ is compact, the set $Q_\Omega (x_k - g_k/\alpha_k)$ is
nonempty for each $k$.
If $y \in Q_\Omega (x_k - g_k/\alpha_k)$, then
since $Q_\Omega (x_k - s_k g_k)$ is set of
elements in $\Omega$ farthest from $x_k - g_k/\alpha_k$, we have
\[
\|y - (x_k - g_k/\alpha_k)\|^2 \ge \|x_k -(x_k - g_k/\alpha_k)\|^2
= \|g_k\|^2/\alpha_k^2 .
\]
Rearrange this inequality to obtain
\begin{equation}\label{3.5}
\nabla f(x_k) (y - x_k) + \frac{\alpha_k}{2}\|y - x_k\|^2 \le 0
\quad \mbox{for all }
y \in Q_\Omega (x_{k}-g_k/\alpha_k).
\end{equation}
Substitute $y = x_{k+1}$ and $x = x_k$ in (\ref{concave3}) and in (\ref{3.5}),
and add to obtain
\begin{eqnarray}
f(x_{k+1}) &\le& f(x_k) +
\left( \frac{\mu - \alpha_k}{2}\right) \|x_{k+1}-x_k\|^2 \nonumber \\
&\le& f_k^{\max} +
\left( \frac{\mu - \alpha_k}{2}\right) \|x_{k+1}-x_k\|^2. \label{h47}
\end{eqnarray}
If $0 > \alpha_k \ge \mu/2$, then $\mu - \alpha_k \le \alpha_k$.
Hence, by (\ref{h47}), Step~2 must terminate whenever $\alpha_k \ge \mu/2$.
Since $\sigma \in (0,1)$, it follows that Step~2 terminates for a finite $j$.
If Step~2 terminates when $j > 0$, then
$\sigma^{j-1} \beta < \mu/2$, which implies that
$\alpha_k = \sigma^{j} \beta < \sigma\mu/2$.
If Step~2 terminates when $j = 0$, then $\alpha_k \le \alpha_{\max} < 0$.
In either case, $\alpha_k \le \bar{\alpha}$.
\end{proof}
\smallskip

In analyzing the convergence of the approximate Newton algorithm,
we also need to consider the continuity of the $Q_\Omega$ operator.
\smallskip

\begin{lemma} \label{continuityQ}
If $\Omega$ is a nonempty compact set, $x_k \in \mathbb{R}^n$ is a
sequence converging to $x^*$ and $y_k \in Q_\Omega (x_k)$ is a
sequence converging to $y^*$, then $y^* \in Q_\Omega (x^*)$.
\end{lemma}
\smallskip

\begin{proof}
Since $y_k \in \Omega$ for each $k$ and $\Omega$ is closed,
$y^* \in \Omega$.
Hence, we have
\[
\|y^* - x^*\| \le \max_{y \in \Omega} \|y - x^*\|.
\]
If this inequality is an equality, then we are done;
consequently, let us suppose that
\[
\|y^* - x^*\| < \max_{y \in \Omega} \|y - x^*\| \le
\max_{y \in \Omega} \{ \|y - x_k \| + \|x_k - x^*\| \} =
\|y_k - x_k \| + \| x_k - x^*\| .
\]
As $k$ tends to $\infty$, the right side approaches
$\|y^* - x^*\|$, which yields a contradiction.
\end{proof}

The following theorem establishes convergence of the approximate
Newton algorithm when $f$ is strongly concave.

\begin{theorem}\label{nonmonotone_theorem}
If $f$ is continuously differentiable on a compact set
$\Omega\subset \mathbb{R}^n$
and $(\ref{concave3})$ holds for some $\mu < 0$, then the sequence of
objective values $f(x_k)$ generated by the nonmonotone approximate
Newton algorithm converge to a limit $f^*$ as $k$ tends to infinity.
If $x^*$ is any accumulation point of the iterates $x_k$, then
$x^* \in Q_\Omega (x^* - g(x^*)/\alpha)$ for some $\alpha < 0$,
which implies that $(\ref{x^*})$ and $(\ref{nc})$ hold.
\end{theorem}
\smallskip

\begin{proof}
By Proposition~\ref{stepsize}, the stepsize generated by the approximate
Newton algorithm satisfies
$\alpha_k \le \bar{\alpha} < 0$.
By the line search criterion and the fact that $\alpha_k < 0$,
$f(x_{k+1}) \le f_k^{\max}$, which implies that
$f_{k+1}^{\max} \le f_k^{\max}$ for all $k$.
Since the $f_k^{\max}$ are monotone decreasing and bounded from below,
there exists a limit $f^*$.
Since $f$ is continuously differentiable on the compact set $\Omega$,
it follows that $f$ is uniformly continuous on $\Omega$.
Since the stepsize $\alpha_k$ is uniformly bounded away from zero,
the same argument used in the proof of \cite[Lem.~4]{wright09a}
shows that $f(x_k)$ converges to $f^*$ and $\|x_{k+1} - x_k\|$ tends to zero.

Select any $L$ satisfying
$\bar{\alpha} < L < 0$ and let $y_k$ be given by
\begin{equation}\label{part0}
y_k \in \arg \; \min \{ \nabla f(x_k) (y - x_k)
+ \frac{L}{2} \|y-x_k\|^2 : y \in \Omega\} .
\end{equation}
It follows that
\begin{eqnarray}
&\nabla f(x_k) (x_{k+1} - x_k) + \frac{L}{2} \|x_{k+1}-x_k\|^2  \ge&
\nonumber \\
&\nabla f(x_k) (y_{k} - x_k) + \frac{L}{2} \|y_{k}-x_k\|^2 .&
\label{part1}
\end{eqnarray}
Since $x_{k+1}$ satisfies (\ref{QP}), we have
\begin{eqnarray}
&\nabla f(x_k) (y_{k} - x_k) + \frac{\alpha_k}{2} \|y_{k}-x_k\|^2  \ge&
\nonumber \\
&\nabla f(x_k) (x_{k+1} - x_k) + \frac{\alpha_k}{2} \|x_{k+1}-x_k\|^2 .&
\label{part2}
\end{eqnarray}
Add (\ref{part1}) and (\ref{part2}) to obtain
\begin{eqnarray*}
&\nabla f(x_k) (x_{k+1} - x_k) + \frac{L}{2} \|x_{k+1}-x_k\|^2  \ge& \\
&\nabla f(x_k) (x_{k+1} - x_k) + \frac{\alpha_k}{2} \|x_{k+1}-x_k\|^2
+ \frac{L-\alpha_k}{2} \|y_k - x_k \|^2.&
\end{eqnarray*}
Since $0 > L > \bar{\alpha} \ge
\alpha_k \ge \alpha_{\min} > -\infty$ and $\|x_{k+1} - x_k\|$ approaches zero,
we conclude that $\|y_k -x_k\|$ tends to zero.
Hence, if $x^*$ is an accumulation point of the iterates $x_k$,
then $x^*$ is an accumulation point of the $y_k$.
Due to the equivalence between (\ref{QP}) and (\ref{max}), it follows
from (\ref{part0}) that $y_k \in Q_\Omega (x_k - g_k/L)$.
Since $g$ is continuous,
$x^* - g(x^*)/L$ is an accumulation point of $x_k - g_k/L$.
By Lemma~\ref{continuityQ}, we have
$x^* \in Q_\Omega (x^* - g(x^*)/L)$, which completes the proof.
\end{proof}

Due to the equivalence between the inclusion
$x^* \in Q_\Omega (x^* - g(x^*))/\alpha)$ and the necessary condition
(\ref{nc}) for a global optimum, the change $E_k = \|x_{k+1} - x_k\|$
associated with the iteration $x_{k+1} \in Q_\Omega (x_k - g_k/\alpha_k)$
could be used to assess convergence of the approximate Newton algorithm.
That is, if $x_{k+1} = x_k$, then $x^* = x_k$ satisfies
the necessary condition (\ref{nc}) for a global optimizer of (\ref{pb}).
As observed in Theorem~\ref{nonmonotone_theorem}, $E_k = \|x_{k+1} - x_k\|$
tends to zero.
We now analyze the convergence rate of $E_k$.
\begin{theorem} \label{rate}
If $f$ is continuously differentiable on a compact set
$\Omega\subset \mathbb{R}^n$
and $(\ref{concave3})$ holds for some $\mu < 0$, then there exists a constant
$c$, independent of $k$, such that
\[
\min \{ E_j : 0 \le j < kM \} \le \frac{c}{\sqrt{k}} .
\]
\end{theorem}
\smallskip

\begin{proof}
Let $\ell (k)$ denote the index associated with the $M$ most recent
function values:
\[
\ell (k) = \arg \; \max \{ f(x_j) : \max\{0, k-M\} < j \le k \} .
\]
In the approximate Newton algorithm, an acceptable step must satisfy the
condition
\begin{equation}\label{accept}
f (x_{k+1}) \le f^{\max}_k + (\alpha_k/2) E_k^2.
\end{equation}
If $k = \ell (j+M) - 1$, then
\begin{equation}\label{j+M}
f(x_{k+1}) = f(x_{\ell(j+M)}) = f_{j+M}^{\max} .
\end{equation}
Since $j < \ell(j+M)$, it follows that
$j-1 < \ell(j+M)- 1 = k$, or $j \le k$.
During the proof of Theorem~\ref{nonmonotone_theorem}, we observed that
the $f_k^{\max}$ sequence is monotone decreasing.
Consequently, $f_k^{\max} \le f_j^{\max}$ since $j \le k$.
Use this inequality and (\ref{j+M}) in (\ref{accept}) to obtain
\begin{equation}\label{fj}
f_{j+M}^{\max} \le f_j^{\max} + (\alpha_k/2) E_k^2,
\end{equation}
where $k = \ell(j+M)- 1$.
Choose $m > 0$ and sum the inequality (\ref{fj})
for $j = 0, M, 2M, \ldots, (m-1)M$.
We have
\begin{equation}\label{f*}
f^* \le f_{mM}^{\max} \le f(x_0) + \sum_{i = 1}^{m}
(\alpha_{k_i}/2) E_{k_i}^2 ,
\end{equation}
where $(i-1)M \le k_i < iM$ and $f^*$ is the limit of the $f_k^{\max}$.
Observe that
\[
\frac{1}{m} \sum_{i=1}^m E_{k_i}^2 \ge
\min_{1 \le i \le m} E_{k_i} \ge
\min_{0 \le i < mM} E_i^2.
\]
Combine this relationship with the bound $\alpha_{k_i} \le \bar{\alpha}$
of Proposition~\ref{stepsize} and (\ref{f*}) to obtain
\[
\min_{0 \le i < mM} E_i^2 \le
\left( \frac{2(f(x_0) - f^*)}{\bar{\alpha}} \right)
\frac{1}{m} .
\]
This completes the proof.
\end{proof}

\section{Numerical experiments}
\label{numerical}
We will investigate the performance of the gradient project and
approximate Newton algorithm relative to previously developed algorithms
in the literature.
In our experiments we use the gradient projection algorithm with
unit stepsize (GPU):
\[
x_{k+1} \in P_\Omega (x_{k}-g_k) .
\]
And in our experiments with the approximate Newton algorithm,
we employ the BB approximation (\ref{BB}) for the
initial stepsize $\beta_k$.
The memory is $M = 50$ and $\sigma = 0.25$.
This version of the approximate Newton algorithm is denoted GPBB.
For the set $\Omega$ associated with sparse PCA, we have
\[
T_\kappa (x)/\|T_\kappa (x)\| \in P_\Omega (x) \quad \mbox{and} \quad
-T_\kappa (x)/\|T_\kappa (x)\| \in Q_\Omega (x)
\]
by Propositions~\ref{projOmega} and \ref{antiprojOmega} respectively.

We compare the performance of our algorithms to those of both
the truncated power method (Tpower) \cite{yuan2013} and the generalized
power method (Gpower) \cite{journee2010}.
The conditional gradient algorithm with unit step-size (ConGradU)
proposed in \cite{luss2013} is equivalent to the truncated power method.
Both truncated and generalized power method are targeted to
the sparse PCA problem (\ref{sparsePCA}).
The truncated power method handles the sparsity constraint by pushing
the absolute smallest components of the iterates to 0.
The iteration can be expressed
\begin{equation}\label{it1}
x_{k+1} = \frac{T_\kappa (-g_k)}{\|T_\kappa(-g_k)\|} .
\end{equation}
For comparison,
an iteration of the gradient projection algorithm with unit step size GPU
is given by
\begin{equation}\label{it2}
x_{k+1} = \frac{T_\kappa (x_k - g_k)}{\|T_\kappa(x_k - g_k)\|} ,
\end{equation}
while the approximate Newton algorithm is
\begin{equation}\label{it3}
x_{k+1} = \mbox{sgn} (\alpha_k) \frac{T_\kappa (x_k - g_k /\alpha_k^{BB})}
{\|T_\kappa(x_k - g_k \alpha_k^{BB})\|},
\end{equation}
where $\mbox{sgn} (\alpha) = 1, 0, -1$ depending on whether
$\alpha > 0$, $\alpha = 0$, or $\alpha < 0$ respectively.
Since the computation of $\alpha_k^{BB}$ requires the gradient at two
points, we start GPBB with one iteration of GPU.
For the sparse PCA problem (\ref{sparsePCA}),
the time for one iteration of any of these methods is basically the time
to multiply a vector by the covariance matrix $\Sigma$.
Note that the monotone approximate Newton algorithm could be more costly
since the evaluation of an acceptable $j$ may require several evaluations
of the objective function.

In the generalized power method, the sparsity constraint is handled using
a penalty terms.
If $\gamma > 0$ denotes the penalty, then Gpower$_{l_1}$ corresponds to
the optimization problem
\[
\max_{\|x\|=1} \sqrt{x^\top \Sigma x}-\gamma\|x\|_1,
\]
where $\|x\|_1 = |x_1| + |x_2| + \ldots + |x_n|$, while
Gpower$_{l_0}$ corresponds to
\[
\max_{\|x\|=1} {x^\top \Sigma x}-\gamma\|x\|_0.
\]
The parameter $\gamma$ needs to be tuned to achieve the desired
cardinality;
as $\gamma$ increases, the cardinality of the Gpower approximation
decreases.
In contrast, the cardinality is an explicit input parameter for
either the truncated power method or for our algorithms;
in many applications, cardinality is often specified.

All experiments were conducted using MATLAB on a GNU/Linux computer
with 8GB of RAM and an Intel Core i7-2600 processor.
For the starting guess in our experiments,
we follow the practice of the Tpower algorithm \cite{yuan2013} and
set $x = e_i$, the $i$-th column of the identity matrix,
where $i$ is the index of the largest diagonal element of
the covariance matrix $\Sigma$.
Our numerical experiments are based on the sparse PCA problem (\ref{sparsePCA}).
We measure the quality of the solution to (\ref{sparsePCA})
computed by any of the methods using
the ratio $x\tr \Sigma x/y\tr \Sigma y$ where $x$ is the sparse first
principal component computed by any of the algorithms for
(\ref{sparsePCA}) and $y$ is the first principal component
(a solution of (\ref{PCA})).
This ratio is often called the {\it proportion of the explained variance}.

\subsection*{Pit props dataset}
This dataset \cite{Jeffers67}
contains 180 observations with 13 variables,
and a covariance matrix $\Sigma\in R^{13\times 13}$.
This is a standard benchmark dataset for Sparse PCA algorithms.
We consider $\kappa = 6$ or 7, and we adjust the value of $\gamma$
to achieve the same sparsity in Gpower.
The last column of Table~\ref{pitp} gives the proportion of the
explained variance.
Observe that all the methods achieve essentially the same proportion
of the explained variance.
\bigskip
\begin{table}[h]
\caption{Results on Pit props data set.}
\label{pitp}
\begin{center}
{\tt
\begin{tabular}{ccc}
\hline
{\bf Method} & {\bf Parameters} & {\bf Explained Variance}\\
\hline
GPBB & $\kappa=6$ & $0.8939$\\
GPBB & $\kappa=7$ & $0.9473$\\
GPU & $\kappa=6$ & $0.8939$\\
GPU & $\kappa=7$ & $0.9473$\\
Tpower(ConGradU) & $\kappa=6$ & $0.8939$\\
Tpower(ConGradU) & $\kappa=7$ & $0.9473$\\
Gpower$_{l_1}$ & $\gamma=0.5$($\Leftrightarrow \kappa =6$) & $0.8939$\\
Gpower$_{l_1}$ & $\gamma=0.4$($\Leftrightarrow \kappa =7$) & $0.9473$\\
Gpower$_{l_0}$ & $\gamma=0.2$($\Leftrightarrow \kappa =6$) & $0.8939$\\
Gpower$_{l_0}$ & $\gamma=0.15$($\Leftrightarrow \kappa =7$) & $0.9473$\\
\hline
\end{tabular}}
\end{center}
\end{table}
\bigskip

We also considered multiple sparse principal components for this
data set and got the same results as those obtained in
Table~2 of \cite{zou06spca} for the Tpower and PathSPCA algorithms.
Similarly, for the lymphoma data set \cite{Alizadeh2000} and
the Ramaswamy data set \cite{ramaswamy01gene}, all methods yielded
the same proportion of explained variance, although the value of
$\gamma$ for Gpower needed to be tuned to achieve the specified cardinality.

\section*{Randomly generated data}
In the next set of experiments, we consider randomly generated data,
where $\Sigma = A\tr A$ with
each entry of $A \in \mathbb{R}^{m \times n}$
generated by a normal distribution with mean 0 and standard deviation 1.
For randomly generated matrices, we can study the performance as either
$m$ or $\kappa$ change.
Each result that we present is based on an average over 100 randomly
generated matrices.
In Figure~\ref{nrm} we plot the proportion of the explained variance
versus cardinality for $m = 250$ and $n = 500$.
Observe that GPBB yields a significantly better objective value
as the cardinality decreases when compared to either GPU or ConGradU, while
GPU and ConGradU have essentially identical performance.
\begin{figure}
\includegraphics[scale=0.6]{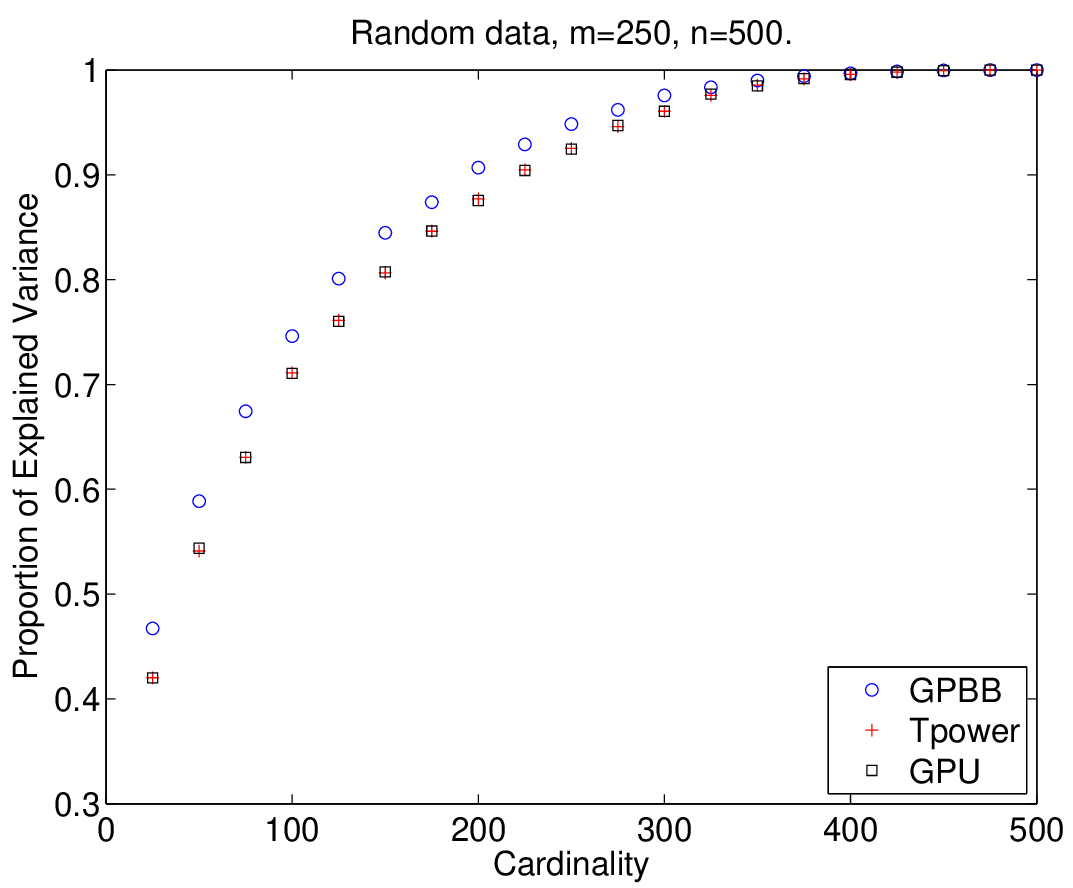}
\caption{Explained variance versus cardinality for random data set.}
\label{nrm}
\end{figure}
Even though all the algorithms seem to yield similar results in
Figure~\ref{nrm} as the cardinality approaches 500,
the convergence of the algorithms is quite different.
To illustrate this, let us consider the case where the cardinality is 500.
In this case where $\kappa = n$,
the sparse PCA problem (\ref{sparsePCA}) and the original
PCA problem (\ref{PCA}) are identical, and the solution of (\ref{PCA})
is a normalized eigenvector associated with the largest eigenvalue
$\lambda_1$ of $\Sigma$.
Since the optimal objective value is known, we can compute the relative error
\[
\frac{\left| \lambda_1^{\mbox{exact}} -\lambda_1^{\mbox{approx}} \right|}
{\left|\lambda_1^{\mbox{exact}}\right|} ,
\]
where $\lambda_1^{\mbox{approx}}$ is the approximation to the optimal
objective value generated by any of the algorithms.
In Figure~\ref{itererror} we plot the base 10 logarithm of the relative error
versus the iteration number.
Observe that GPBB is able to reduce the relative error to the machine
precision near $10^{-16}$ in about 175 iterations,
while ConGradU and GPU have relative error around $10^{-3}$
after 200 iterations.
To achieve a relative error around $10^{-16}$, ConGradU and GPU
require about 4500 iterations, roughly 25 times more than GPBB.

Despite the very nonmonotone behavior of the GPBB iterates,
the convergence is relatively fast.
The results for the explained variance in Figure~\ref{nrm}
were obtained by running either ConGradU or GPU for 6000 iterations,
while GPBB was run for 200 iterations.
Hence, the better objective values obtained by GPBB in Figure~\ref{nrm}
were due to the algorithm converging to a better solution,
rather than to premature termination of either GPU or ConGradU.
\begin{figure}
\includegraphics[scale=0.6]{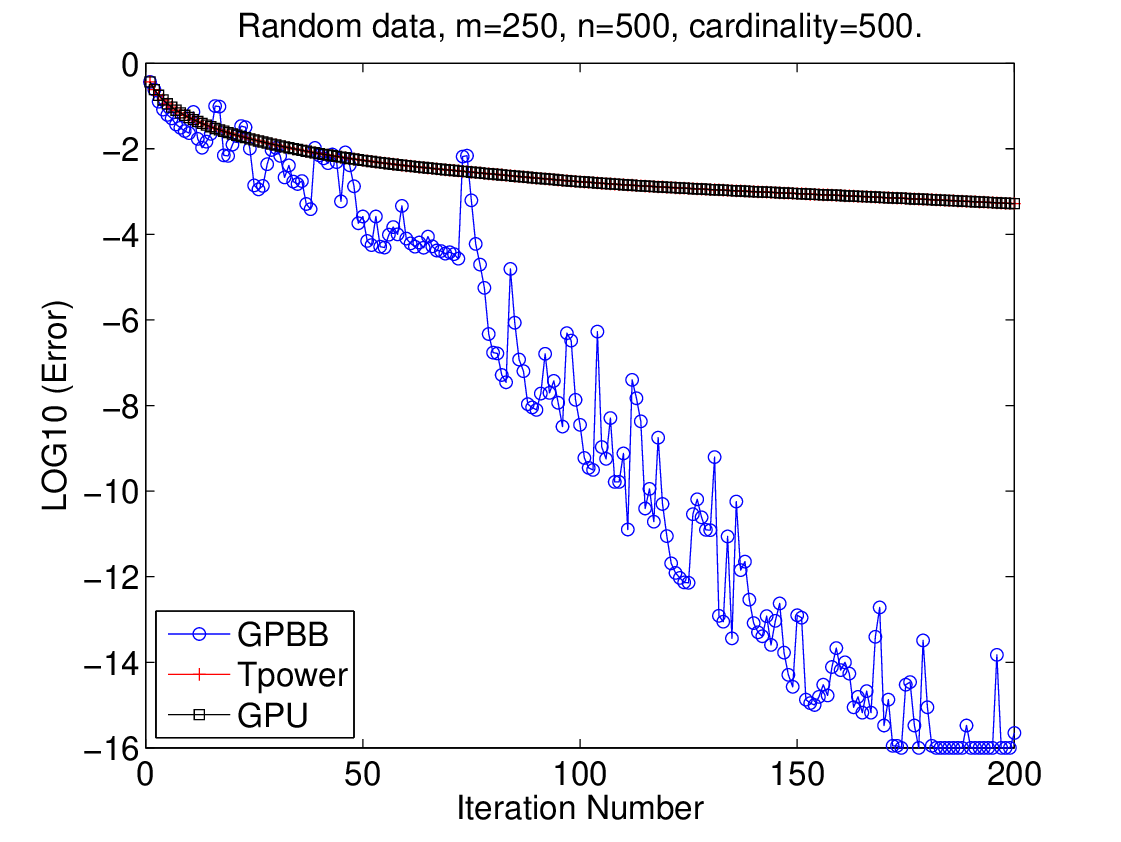}
\caption{A plot of the base 10 logarithm of the relative error versus
iteration number for the random data set with $m = 250$ and cardinality
$\kappa = 500$.}
\label{itererror}
\end{figure}

In Figure~\ref{randiter} we plot the proportion of the explained variance
versus the iteration number when $m = 250$, $n = 500$,
and the cardinality $\kappa = 50$ in the random data set.
When we plot the function value as in Figure~\ref{randiter},
it is more difficult to see the nonmonotone nature of the convergence for GPBB.
This nonmonotone nature is clearly visible in Figure~\ref{itererror}
where we plot the error instead of the function value.
In Figure~\ref{mval} we show how the explained variance depends on $m$.
As $m$ increases, the explained variance associated with GPBB becomes
much better than that of either ConGradU or GPU.
\begin{figure}
\includegraphics[scale=0.6]{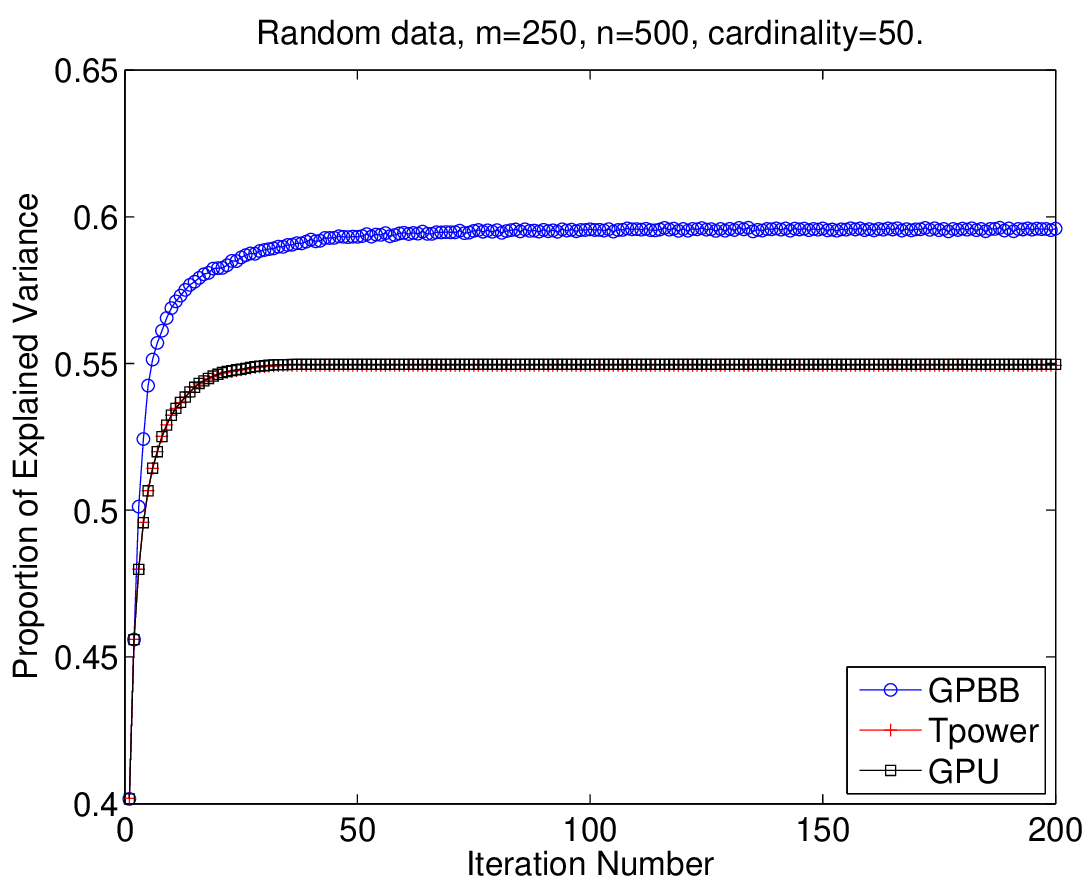}
\caption{Explained variance versus iteration number for cardinality 50 in
the random data set.}
\label{randiter}
\end{figure}
\begin{figure}
\includegraphics[scale=0.6]{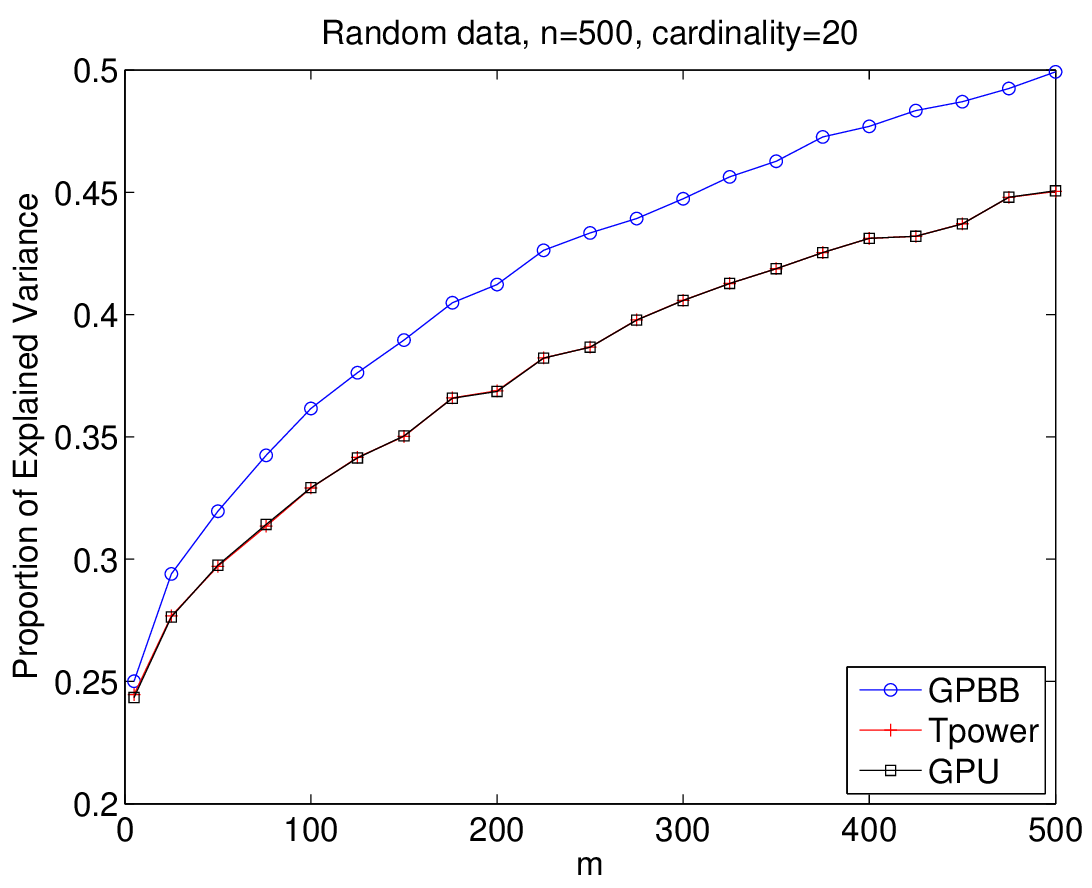}
\caption{Explained variance versus $m$ for cardinality 20 in
the random data set.}
\label{mval}
\end{figure}

In Figure~\ref{ls} we compare the relative error of GPBB for various choices
of the memory $M$.
Observe that the monotone scheme where $M = 1$ is slowest, while
$M = 50$ gives better performance than that of either a small $M$ or
$M = 0$ where $x_{k+1} \in Q_\Omega (x_k -g_k/\alpha_k^{BB})$.
Note that the running time of the monotone scheme
was about 4 times larger than that of the nonmonotone schemes
since we may need to test several choices of
$\alpha_k$ before generating a successful monotone iterate.
\begin{figure}
\includegraphics[scale=0.45]{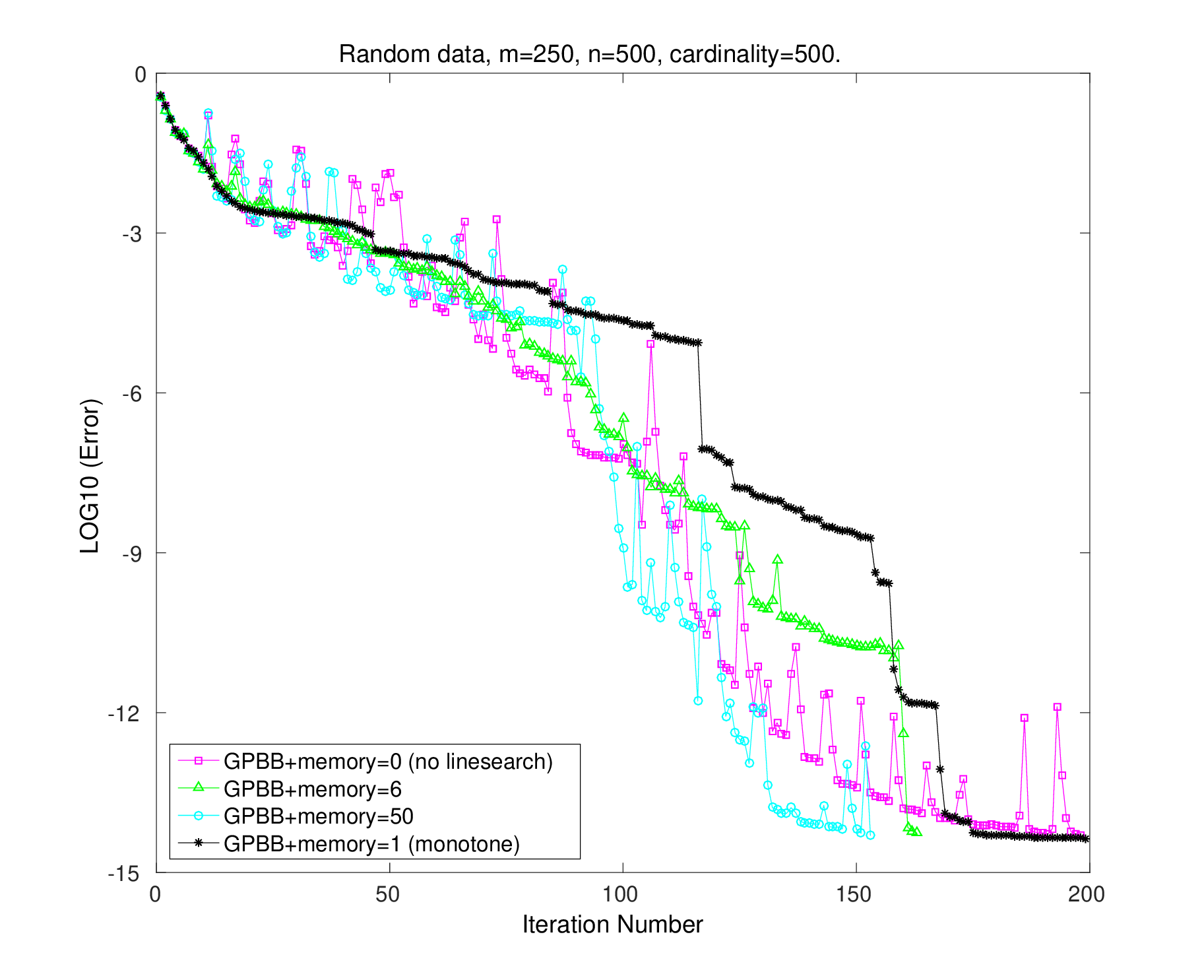}
\caption{A plot of the base 10 logarithm of the relative error versus
iteration number for the random data set with $m = 250$, cardinality
$\kappa = 500$, and various choices for the memory in GPBB.}
\label{ls}
\end{figure}\\

To compare with Gpower, we need to choose a value for $\gamma$.
We first consider a simple case $m=20, n=20$ and we use the ``default''
seed in MATLAB to generate this matrix.
The algorithms are used to extract the first principal component
with $\kappa = 5$, and with $\gamma$ tuned to achieve $\kappa = 5$.
The results in Table~\ref{simpletable} indicate that Gpower performed
similar to Tpower, but not as well as GPBB.
\bigskip
\begin{table}[t]
\caption{Simple random dataset}
\label{simpletable}
\begin{center}
{\tt
\begin{tabular}{ccc}
\hline
{\bf Method} & {\bf Cardinality} & {\bf Explained Variance}\\
\hline
GPBB & $\kappa=5$ & $0.8193$\\
Tpower(ConGradU) & $\kappa=5$ & $0.7913$\\
Gpower$_{l_1}$ & $\gamma=0.18$($\Leftrightarrow \kappa =5$) & $0.7914$\\
Gpower$_{l_0}$ & $\gamma=0.045$($\Leftrightarrow \kappa =5$) & $0.7914$\\
\hline
\end{tabular}}
\end{center}
\end{table}

In the next experiment, we consider 100 randomly generated matrices
with $m = 250$ and $n = 500$, and with
the parameter $\gamma$ in Gpower chosen to achieve an average
cardinality near 100 or 120.
As seen in Table~\ref{randtable},
Gpower$_{l_0}$ achieves similar values for the proportion of the
explained variance as Tpower,
while Gpower$_{l_1}$ achieves slightly better
results and GPBB achieves the best results.
\bigskip

\begin{table}[h]
\caption{Random data set, $m = 250$, $n = 500$.}
\label{randtable}
\begin{center}
{\tt
\begin{tabular}{ccc}
\hline
{\bf Method} & {\bf Cardinality } & {\bf Explained Variance}\\
\hline
GPBB & $\kappa =100$ & $0.7396$\\
GPBB & $\kappa =120$ & $0.7823$\\
Tpower(ConGradU) & $\kappa =100$ & $0.7106$\\
Tpower(ConGradU) & $\kappa =120$ & $0.7536$\\
Gpower$_{l_1}$ & $\gamma=0.075$(average $\kappa =99$) & $0.7288$\\
Gpower$_{l_1}$ & $\gamma=0.0684$(average $\kappa =120$) & $0.7679$\\
Gpower$_{l_0}$ & $\gamma=0.0078$(average $\kappa =100$) & $0.7129$\\
Gpower$_{l_0}$ & $\gamma=0.0066$(average $\kappa =120$) & $0.7557$\\
\hline
\end{tabular}}
\end{center}
\end{table}
\bigskip

\subsection*{Hollywood-2009 dataset, Densest $k$-subgraph (DkS)}
Given an undirected graph $G = (\C{V}, \C{E})$ with vertices
$\C{V} = \{1, 2, \ldots, n\}$ and edge set $\C{E}$,
and given an integer $k \in [1, n]$,
the densest $k$-subgraph (DkS) problem is to find a set of $k$ vertices
whose average degree in the subgraph induced by this set is as large as
possible.
Algorithms for finding DkS are useful tools for analyzing networks.
Many techniques have been proposed for solving this problem
including \cite{bhaskara2010}, \cite{khuller2009}, \cite{YeZhang2003}.
Mathematically, DkS is equivalent to a binary quadratic programming problem
\begin{equation} \label{DKS}
\max \{ \pi  \tr A \pi :
\pi \in \mathbb{R}^n, \quad \pi\in \{0,1\}^n, \quad \|\pi\|_0 = k \} ,
\end{equation}
where $A$ is the adjacency matrix of the graph;
$a_{ij} = 1$ if $(i,j) \in \C{E}$, while $a_{ij} = 0$ otherwise.
We relax the constraints
$\pi\in \{0,1\}^n$ and $\|\pi\|_0 = k$ to
$\|\pi\| = \sqrt{k}$ and $\|\pi\|_0 \le k$, and
consider the following relaxed version of (\ref{DKS})
\begin{equation} \label{DKS_relax}
\max \{ \pi  \tr A \pi : \pi \in \mathbb{R}^n, \quad \|\pi\| = \sqrt{k},
\quad \|\pi\|_0 \le k \} .
\end{equation}
After a suitable scaling of $\pi$, this
problem reduces to the sparse PCA problem (\ref{sparsePCA}).

Let us consider the Hollywood-2009 dataset \cite{BRSLLP, BoVWFI},
which is associated with a graph
whose vertices are actors in movies, and an edge joins two vertices
whenever the associated actors appear in a movie together.
The dataset can be downloaded from the following web site:
\smallskip
\begin{center}
http://law.di.unimi.it/datasets.php
\end{center}
\smallskip
The adjacency matrix $A$ is $1139905\times 1139905$.
In order to apply Gpower to the relaxed problem, we first factored
$A + cI$ into a product of the form $R\tr R$ using a Cholesky factorization,
where $c > 0$ is taken large enough to make $A + cI$ positive definite.
Here, $R$ plays the role of the data matrix.
However, one of the steps in the Gpower code updates the data matrix
by a rank one matrix, and the rank one matrix caused the
updated data matrix to exceed the 200~GB memory on the largest
computer readily available for the experiments.
Hence, this problem was only solved using Tpower and GPBB.
Since the adjacency matrix requires less than 2~GB memory, it easily
fit on our 8~GB computer.

In Table~\ref{hollywoodtbl1}, we compare the density values $\pi\tr A\pi/k$
obtained by the algorithms.
In addition, we also computed the largest eigenvalue
$\lambda$ of the adjacency matrix $A$, and give the ratio of the
density to $\lambda$.
Observe that in 2 of the 6 cases, GPBB obtained a significantly better
value for the density when compared to Tpower, while in the
other 4 cases, both algorithms converged to the same maximum.

\bigskip
\begin{table}[h]
\caption{Hollywood data set.}
\label{hollywoodtbl1}
\begin{center}
{\tt
\begin{tabular}{clcc}
\hline
{\bf Method} & \multicolumn{1}{c}{\bf Cardinality}
& {\bf Density $\pi\tr A\pi/k$}
& {\bf Ratio $\displaystyle\frac{\pi\tr A\pi/k}{\lambda}$}\\
\hline
GPBB & \quad $k=500$ & $379.40$ & $0.1688$\\
GPBB & \quad $k=600$ & $401.22$ & $0.1785$\\
GPBB & \quad $k=700$ & $593.24$ & $0.2639$\\
GPBB & \quad $k=800$ & $649.67$ & $0.2891$\\
GPBB & \quad $k=900$ & $700.38$ & $0.3116$\\
GPBB & \quad $k=1000$ & $745.95$ & $0.3319$\\
Tpower(ConGradU) & \quad $k=500$ & $190.11$ & $0.0846$\\
Tpower(ConGradU) & \quad $k=600$ & $401.21$ & $0.1785$\\
Tpower(ConGradU) & \quad $k=700$ & $436.53$ & $0.1942$\\
Tpower(ConGradU) & \quad $k=800$ & $649.67$ & $0.2891$\\
Tpower(ConGradU) & \quad $k=900$ & $700.44$ & $0.3116$\\
Tpower(ConGradU) & \quad $k=1000$ & $745.95$ & $0.3319$\\
\hline
\end{tabular}}
\end{center}
\end{table}
\bigskip

%
%

\section{Conclusion}
\label{conclusion}
The gradient projection algorithm was studied in the case where
the constraint set $\Omega$ may be nonconvex,
as it is in sparse principal component analysis.
Each iteration of the gradient projection algorithm
satisfied the condition $\nabla f(x_k)(x_{k+1}-x_k) \le 0$.
Moreover, if $f$ is concave over $\mbox{conv}(\Omega)$, then
$f(x_{k+1}) \le f(x_k)$ for each $k$.
When a subsequence of the iterates converge to $x^*$, we obtain
in Theorem~\ref{converge} the equality $\nabla f(x^*) (y^* - x^*) = 0$
for some $y^* \in P_\Omega (x^* - s^* g(x^*))$ where $P_\Omega$ projects
a point onto $\Omega$ and $s^*$ is the limiting step size.
When $\Omega$ is convex, $y^*$ is unique and the condition
$\nabla f(x^*) (y^* - x^*) = 0$ is equivalent to the first-order
necessary optimality condition at $x^*$ for a local minimizer.

The approximate Newton algorithm with a positive definite
Hessian approximation $\alpha_k$ reduced to the projected gradient algorithm
with step size $1/\alpha_k$.
On the other hand, when $\alpha_k < 0$,
as it is when the objective function is concave and the Hessian
approximation is computed by the Barzilai-Borwein formula (\ref{BB}),
the iteration amounts to taking a step along the positive
gradient, and then moving as far away as possible while staying inside the
feasible set.
In numerical experiments based on sparse principal component analysis,
the gradient projection algorithm with
unit step size performed similar to both the truncated power method
and the generalized power method.
On the other hand, in some cases, the approximate Newton algorithm
with a Barzilai-Borwein stepsize and a nonmonotone line search
could converge much faster to a better objective function value
than the other methods.
\bibliographystyle{siam}
\bibliography{library}
\end{document}